\documentclass[oneside,12pt]{amsart}
\usepackage{eurosym}
\usepackage{amsfonts} 
\usepackage{amsmath,amssymb,amsthm,xcolor}
\usepackage{latexsym}
\usepackage{float}
\usepackage{placeins}
\usepackage{bm}
\usepackage[utf8]{inputenc}
\usepackage[british]{babel}
\usepackage{a4wide}
\usepackage{pdfsync}
\usepackage[colorlinks=true,linkcolor=blue,citecolor=blue]{hyperref}
\usepackage[all]{xy}
\usepackage{parskip}
\usepackage{multirow}
\usepackage{array}
\usepackage[toc,page]{appendix}
\usepackage{graphicx}
\usepackage{comment}
\usepackage[mathscr]{euscript}
\usepackage{enumitem}
\usepackage{orcidlink}

\theoremstyle{plain}

\newtheorem{maintheorem}{Theorem}

\newtheorem*{problem}{Problem}

\newtheorem{lemma}{Lemma} 

\newtheorem*{theorem*}{Theorem}

\theoremstyle{definition} 

\newtheorem*{definition*}{Definition}
\newtheorem{remark}[lemma]{Remark}
\newtheorem{example}[lemma]{Example} 

\theoremstyle{remark}

\newcommand{\produ}[1]{\langle #1\rangle}
\newcommand{\R}{\mathbb{R}}

\renewcommand{\div}{\mathrm{div}}

\newcommand{\Tr}{\mathrm{Tr}}
\newcommand{\Scal}{\mathrm{Scal}}
\newcommand{\Ric}{\mathrm{Ric}}
\newcommand{\Id}{\mathrm{Id}}

\begin{document}

\title[A class of GKS determined by the Ricci tensor and the metric]{A class of generalised Killing spinors determined by the Ricci tensor and the metric}

 \author[Diego Artacho]{Diego Artacho \orcidlink{0000-0003-2345-5043}} 
\author[Jihun Kim]{Jihun Kim \orcidlink{0000-0001-5501-1207}}

\address{Department of Mathematics, KU Leuven, Celestijnenlaan 200B, 3001 Leuven, Belgium}
\email{diego.artachodeobeso@kuleuven.be, jihun.kim@kuleuven.be}

\thanks{This work has been partially supported by FWO and FNRS under EOS project G0I2222N} 

\subjclass[2020]{53C27, 53C25, 53C21, 22E60}
\keywords{Generalised Killing spinors, special spinors, Ricci curvature}

\begin{abstract}
We introduce a class of generalised Killing spinors, termed affine Killing spinors (AKS), for which the associated endomorphism is a constant linear combination of the Ricci endomorphism and the identity map. We classify Riemannian spin manifolds admitting an AKS under two additional curvature hypotheses: harmonic curvature and local conformal flatness. Furthermore, we characterise Riemannian spin manifolds that admit a non-zero parallel one-form and an AKS. Additionally, we prove that in dimension three every curvature-homogeneous manifold carrying an AKS is locally homogeneous. Finally, we provide a complete classification of three-dimensional Lie groups equipped with a Bianchi metric admitting an invariant AKS.  
\end{abstract}

\maketitle

\section{Introduction and main results}\label{s:int}

Let $(M,g)$ be a Riemannian manifold admitting a spin structure, and let $\Sigma M$ be its spinor bundle. A non-zero section $\psi \in \Gamma(\Sigma M)$ is said to be a \emph{generalised Killing spinor} (GKS) if there exists a $g$-symmetric endomorphism field $A \in \mathrm{End}(TM)$ such that 
\begin{equation*}\label{eq:gks}
\nabla_X \psi = A(X) \cdot \psi , \qquad \text{for all } X \in \Gamma(TM) , 
\end{equation*}
where $\nabla$ is the connection on $\Sigma M$ induced by the Levi-Civita connection of $g$ and $\cdot$ denotes Clifford multiplication. 
It is clear that a GKS necessarily has constant length with respect to the Hermitian inner product induced by $g$ on $\Sigma M$. These spinors arise naturally as restrictions of parallel spinors to hypersurfaces, with the corresponding endomorphism given by one half of the shape operator~\cite{BGM,AMM}. 

The existence of a GKS as an intrinsic object imposes restrictions on the geometry of the manifold. For instance, if $A = \mu \, \Id$ for some constant $\mu \in \mathbb{C}$, $\psi$ is said to be a \emph{Killing spinor} (KS), and the existence of such a spinor forces $g$ to be an Einstein metric~\cite{BFGK}. Besides the identity map, one of the simplest geometrically defined symmetric endomorphisms on a given Riemannian manifold is the Ricci endomorphism, denoted by $\Ric$. Recently, GKS with associated endomorphism $A = \lambda \, \Ric$ for a non-zero constant $\lambda \in \R$, called \emph{Ricci Killing spinors} (RKS), have been studied by Imada~\cite{I26}. It is therefore natural to investigate a wider class of GKS, which is in some sense the affine closure of the class of KS and RKS: 

\begin{definition*}
    On a Riemannian spin manifold $(M,g)$, an \emph{affine Killing spinor} (AKS) is a non-zero spinor $\psi$ that satisfies 
     \begin{equation}\label{eq:ars}
         \nabla_X \psi = (\lambda \, \Ric + \mu \, \Id )(X) \cdot \psi
     \end{equation}
    for all $X\in \Gamma(TM)$ and some constants $\lambda\in\R \setminus \{0\} ,\,\mu \in \mathbb{R}$.
\end{definition*}

A similar class of special spinors was introduced in~\cite[Def.~3.1]{KF}: on a spin $n$-manifold with $n \ge 3$ and nowhere-vanishing scalar curvature $\Scal$, a spinor $\psi$ is said to be a \emph{weak Killing spinor} (WKS) if it satisfies the equation
\[
\nabla_X \psi = \tfrac{n}{2(n-1)\Scal}{\rm d}\Scal(X)\psi + \tfrac{1}{2(n-1)\Scal}X\cdot {\rm d}\Scal \cdot \psi + \left\{\tfrac{2\sigma}{(n-2)\Scal}\Ric(X) - \tfrac{\sigma}{n-2}X\right\}\cdot\psi
\]
for all $X\in \Gamma(TM)$ and some constant $\sigma\in \mathbb{R}$.
In the special case where the scalar curvature is constant, WKS are AKS with $2 \mu = - \lambda \, \Scal$ in~\eqref{eq:ars}. 
This motivates the extension of several results of~\cite{KF} concerning WKS in the non-zero constant scalar curvature regime to the broader class of AKS. 

In what follows, we present our main results and provide the context behind them.

Our first result establishes strong geometric restrictions imposed by the existence of an AKS in the settings of harmonic curvature and local conformal flatness:

\begin{maintheorem}\label{thm:main}
    Let $(M,g)$ be a Riemannian spin manifold admitting an AKS.
    \begin{enumerate}[label={\rm{(\alph*)}}]
        \item\label{it:har} If $(M,g)$ has harmonic curvature, then it is either Einstein {\rm(}and hence the spinor is a KS{\rm)} or locally isometric to a Riemannian product of an Einstein manifold and a Ricci-flat manifold. Moreover, if the Ricci-flat factor has dimension one, then the spinor is an RKS.  
        \item\label{it:lcf} If $(M,g)$ is locally conformally flat, then it is locally isometric to either a space form or a Riemannian product of a space form and a real line. 
    \end{enumerate}
\end{maintheorem}

We recall that a Riemannian manifold has harmonic curvature if and only if its Ricci endomorphism is Codazzi. The proof of Theorem~\ref{thm:main}\ref{it:har} shows that, assuming $\Ric$ is Codazzi, the existence of an AKS implies that $\Ric$ has at most two distinct eigenvalues, and that they are constant. It is well known that the eigendistributions of a Codazzi endomorphism with constant eigenvalues are integrable and their leaves are totally geodesic~\cite[Rem. 1.2]{D81}.  We then show that the existence of an AKS also allows one to upgrade this to a local Riemannian product decomposition. 

Besides being a natural class of GKS, another reason to study AKS is that, by Lemma~\ref{lem:AKSbasic}, their existence forces not only the scalar curvature but also the squared norm of the Ricci endomorphism to be constant.
These are hypotheses that often appear in the literature: for instance, Cheng--Ishikawa--Shiohama~\cite{CIS01} and He--Li~\cite{HL18} proved the non-existence of \emph{compact} locally conformally flat Riemannian $n$-manifolds ($n \ge 3$) with constant scalar curvature and constant norm of the Ricci endomorphism, provided that the Ricci endomorphism has exactly three distinct eigenvalues \emph{at every point}. In contrast, in this article we establish a non-existence result for locally conformally flat spin $n$-manifolds ($n \ge 3$, \emph{not necessarily compact}) admitting an AKS under the assumption that the Ricci endomorphism has \emph{at least} three distinct eigenvalues \emph{at some point} -- see Theorem~\ref{thm:main}\ref{it:lcf}.

We now turn to our next result.
For an $n$-dimensional Riemannian spin manifold $M$ admitting an AKS $\psi$ satisfying~\eqref{eq:ars}, the Dirac operator $D$ acts on $\psi$ as 
\[D\psi = \sum\limits_{i=1}^{n}e_i\cdot \nabla_{e_i}\psi= -\Tr(\lambda\Ric + \mu\Id)\psi= -(\lambda\Scal+n\mu)\psi,\]
where $(e_1,\dots,e_n)$ is a local orthonormal frame of $M$.
Hence, any AKS $\psi$ is an eigenspinor of $D$ with eigenvalue $-\lambda\Scal-n\mu$. The problem of determining optimal lower bounds for the eigenvalues of the Dirac operator on compact manifolds was first considered by Friedrich in 1980~\cite{Fr}: every eigenvalue $\rho$ of $D$ satisfies $\rho^2 \ge \frac{n}{4(n-1)}\inf_{M}\Scal$. He also showed that Killing spinors are precisely those eigenspinors of $D$ that realise the lowest possible eigenvalue.
Since then, considerable effort has been devoted to improving this inequality under various geometric assumptions, including conditions on the scalar curvature and the existence of a parallel one-form -- see e.g.~\cite{AGI,Ch}.

In~\cite[Thm.~4.5]{KF}, it was shown that a Riemannian manifold with non-zero constant scalar curvature and a non-zero parallel one-form cannot admit a WKS. Since, as noted above, a Riemannian spin manifold with a WKS and non-zero constant scalar curvature is a special case of a manifold admitting an AKS, we consider manifolds admitting an AKS together with a non-zero parallel one-form: 
\begin{maintheorem}\label{thm:p1form}
    Let $(M,g)$ be an $n$-dimensional Riemannian spin manifold admitting an AKS as in~\eqref{eq:ars}, and suppose that $M$ admits a non-zero parallel one-form. Then, exactly one of the following holds:
     \begin{enumerate}[label={\rm{(\alph*)}}]
        \item\label{it:p1form1} $(M,g)$ is Ricci-flat and the spinor is parallel.
        \item\label{it:p1form2} $\Scal \neq 0$, the spinor is an RKS, and $\lambda^2 > \frac{n}{4(n-1)\Scal}$.
        \item\label{it:p1form3} $n \neq 3$, $\Scal = \frac{\mu}{\lambda}(1-n) \neq 0$, and $\mu^2 > \frac{n}{4(n-1)}\Scal$.
    \end{enumerate}
\end{maintheorem}

Note that, in all three cases, the Dirac eigenvalue of the AKS satisfies Friedrich's bound, even without assuming compactness. Moreover, the inequality is strict whenever $(M,g)$ is not Ricci-flat.

Additionally, we show that the existence of an AKS is related to homogeneity properties of the manifold. Curvature homogeneity, a notion introduced by Singer~\cite{S60}, is a weaker condition than local homogeneity. A Riemannian manifold $(M,g)$ is said to be \emph{curvature homogeneous} if, for every pair of points $p,q\in M$, there exists a linear isometry $\Phi\colon T_pM \to T_qM$ such that $\Phi^*\mathrm{R}_q=\mathrm{R}_p$, where $\mathrm{R}$ denotes the Riemann curvature tensor. In dimension $2$, curvature homogeneity forces the Gaussian curvature to be constant; consequently, a curvature-homogeneous surface is necessarily locally homogeneous. In higher dimensions, however, there are examples of curvature-homogeneous manifolds that are not locally homogeneous. The first one was given by Takagi~\cite{T74} in dimension $3$. Since then, numerous examples have been discovered. We refer the reader to~\cite{BKV} and the references therein for further details. On the other hand, we also note that, under suitable additional assumptions, curvature homogeneity may imply local homogeneity -- see e.g.~\cite{Sek}. We prove that, in dimension $3$, the existence of an AKS allows one to conclude local homogeneity from curvature homogeneity:

\begin{maintheorem}\label{thm:locallyhom}
    Let $(M,g)$ be a connected and oriented curvature-homogeneous Riemannian $3$-manifold.
    If $M$ admits an AKS, then it is locally homogeneous. 
\end{maintheorem}

Notice that, in Theorem~\ref{thm:locallyhom}, the assumption of curvature homogeneity can be weakened to the requirement that one of the eigenvalues of $\Ric$ be constant (this follows from Lemma~\ref{lem:AKSbasic}). 

Recently, Imada classified RKS on $3$-dimensional Lie groups~\cite{I26}. We conclude the article by giving a complete classification of invariant AKS on $3$-dimensional Lie groups equipped with a left-invariant metric, providing a large family of examples: 

\begin{maintheorem}\label{thm:mainlie}
    Let $G$ be a connected and oriented $3$-dimensional Lie group with Lie algebra $\mathfrak{g}$. Then, $G$ admits a left-invariant metric carrying an invariant AKS that is not an RKS if and only if $\mathfrak{g}$ is either $\mathfrak{su}(2)$ or $\mathfrak{sl}(2,\mathbb{R})$. 
\end{maintheorem}

Finally, we would like to remark that, in all examples of connected Riemannian spin manifolds carrying an AKS known to the authors, the Ricci endomorphism has at most two distinct eigenvalues, and therefore these eigenvalues are constant (as $\Scal$ and $\|\Ric\|^2$ are constant). In particular, by Theorem~\ref{thm:locallyhom}, all known oriented Riemannian $3$-manifolds admitting an AKS are locally homogeneous. This motivates the following problem:  

\begin{problem} 
Do there exist oriented Riemannian $3$-manifolds admitting an AKS that are not locally homogeneous?
\end{problem}

\section{Preliminaries}
We recall the basic notation for Riemannian spin manifolds and collect several facts that will be needed in the sequel when the manifold admits an affine Killing spinor; see~\cite{BFGK,MS141} for further details.

\subsection{Spin manifolds}
Let $(M,g)$ be an $n$-dimensional Riemannian manifold with a spin structure, and let $\Sigma M$ denote its spinor bundle. The metric $g$ induces a Hermitian metric $\produ{\cdot,\cdot}$ on $\Sigma M$. The Levi-Civita connection $\nabla$ of $g$ induces a covariant derivative on $\Sigma M$ which we also denote by $\nabla$, and this is compatible with $\produ{\cdot,\cdot}$. Clifford multiplication with tangent vector fields is parallel with respect to $\nabla$ and skew-Hermitian with respect to $\produ{\cdot,\cdot}$, that is, $\produ{X\cdot \psi,\phi} = - \produ{\psi , X\cdot\phi}$ for every $X \in \Gamma(TM)$ and $\psi,\phi\in\Gamma(\Sigma M)$. In particular, $\Re \produ{X\cdot\psi,\psi}=0$, where $\Re$ denotes real part.

We can define Clifford multiplication with $2$-forms as
\begin{equation*}\label{eq:cli2f}
    (X\wedge Y) \cdot \psi := X\cdot Y\cdot \psi + g(X,Y)\psi.
\end{equation*}
On the other hand, the skew-symmetry of Clifford multiplication and the basic Clifford identity $X\cdot Y + Y\cdot X + 2g(X,Y) = 0$ give us
\begin{equation*}\label{eq:clieq1}
   \Re \produ{X\cdot Y\cdot \psi,\psi} = -g(X,Y)\produ{\psi,\psi}
\end{equation*}
for all $X,Y\in \Gamma(TM)$ and $\psi\in\Gamma(\Sigma M)$. 

We denote by $\mathrm{R}$ the Riemann curvature tensor of $M$, defined by $\mathrm{R}_{X,Y}:=[\nabla_{X},\nabla_{Y}]-\nabla_{[X,Y]}$. The Ricci tensor, the Ricci endomorphism, and the scalar curvature are denoted by $\mathrm{ric}$, $\Ric$, and $\Scal$, respectively, where $\mathrm{ric}(X,Y):=g(\Ric(X),Y)$.
The curvature $\mathrm{R}^{\Sigma M}$ of the spinor bundle $\Sigma M$ is related to the Riemann curvature of $M$ by the following:
\begin{equation}\label{eq:spincurv}
    \mathrm{R}^{\Sigma M}_{X,Y} \psi = \frac{1}{2}\mathcal{R}(X\wedge Y)\cdot \psi,
\end{equation}
where $\mathcal{R}:\Lambda^2TM \to \Lambda^2 TM$ denotes the curvature operator on $2$-forms given as 
\[
g(\mathcal{R}(X\wedge Y),Z\wedge W):= g(\mathrm{R}_{X,Y}Z,W).
\]
With this convention, the curvature operator $\mathcal{R}$ of a space of constant sectional curvature $\kappa$ acts on $2$-forms as $-\kappa\,\Id$.

\subsection{Affine Killing spinors} 
We record several lemmas that will be used later.
\begin{lemma}[{\cite{MS141}}]\label{lem:basiclem}
       Let $(M,g)$ be an $n$-dimensional Riemannian spin manifold admitting a GKS $\psi$ with associated endomorphism $A$. Then, the following equations hold: 
       \begin{enumerate}[label={{\rm(\alph*)}}]
         \item\label{it:b1} $[(\nabla_X A)Y - (\nabla_Y A)X]\cdot \psi = 2 (A(X)\wedge A(Y)) \cdot \psi +\frac{1}{2}\mathcal{R}(X\wedge Y) \cdot \psi$,
         \item\label{it:b2} $\left(\sum_{i=1}^{n}e_i \wedge (\nabla_{e_i}A)X \right)\cdot \psi = \left(\frac{1}{2}\Ric(X) + 2 A^2(X) - 2\,\Tr(A) A(X)\right)\cdot\psi$,
         \item\label{it:b3} $\div A + {\rm d}(\Tr(A)) = 0$,
         \item\label{it:b4} $(\Tr A)^2 - \Tr A^2 = \frac{1}{4}\Scal$. 
       \end{enumerate} 
Here, $\div A$ denotes the divergence of $A$, and is given by $\div A := - \sum_{i=1}^n(\nabla_{e_i} A)e_i$ for a local orthonormal frame $(e_1 , \dots , e_n)$. \qed
\end{lemma}

Substituting $A = \lambda \Ric + \mu\Id$ in Lemma~\ref{lem:basiclem}, we obtain the following: 

\begin{lemma}\label{lem:AKSbasic}
Let $(M,g)$ be an $n$-dimensional Riemannian spin manifold admitting an AKS $\psi$ satisfying~\eqref{eq:ars}. Then, the scalar curvature $\Scal$ and $\|\Ric\|^2$ are constant, and the following identities hold:

\begin{equation}\label{eq:bas1}
    \begin{aligned}
       &\lambda[(\nabla_X \Ric)Y - (\nabla_Y \Ric)X]\cdot \psi\\
       &= 2[\lambda^2\Ric X\wedge \Ric Y + \lambda\mu(\Ric X\wedge Y + X \wedge \Ric Y)+\mu^2 X\wedge Y] \cdot \psi+\frac{1}{2}\mathcal{R}(X\wedge Y) \cdot \psi,
    \end{aligned}
\end{equation}
\begin{equation}\label{eq:bas2}
    \begin{aligned}
        \lambda\Big(\sum_{i=1}^{n}e_i \wedge (\nabla_{e_i}\Ric)X\Big)\cdot \psi  &= 2\Big\{\lambda^2\Ric^2(X) + (\tfrac{1}{4}+(2-n)\lambda\mu -\lambda^2\Scal)\Ric(X)\\
        &~~~~~~~~~~~ +((1-n)\mu^2-\lambda\mu\Scal)X\Big\}\cdot\psi,
    \end{aligned}   
\end{equation}
\begin{equation}\label{eq:bas3}
    \|\Ric\|^2 = \Scal^2 + \tfrac{1}{\lambda^2}(2\lambda\mu(n-1)-\tfrac{1}{4})\Scal + n(n-1)(\tfrac{\mu}{\lambda})^2.
\end{equation} 
\end{lemma}
\begin{proof}
    Equations~\eqref{eq:bas1},~\eqref{eq:bas2} and~\eqref{eq:bas3} follow directly from (a), (b) and (d) of Lemma~\ref{lem:basiclem} respectively. Point (c) yields that $\lambda (\div (\Ric) + {\rm d} \Scal) = 0$. As $\lambda \neq 0$ and $\div (\Ric) = -\frac{1}{2} {\rm d} \Scal$ by the second Bianchi identity, we conclude that $\Scal$ is constant on $M$. By~\eqref{eq:bas3}, $\|\Ric\|^2$ is also constant. 
\end{proof}

\begin{remark}
    Notice that Lemma~\ref{lem:AKSbasic} determines Riemannian spin manifolds of dimension at most two admitting an AKS: in dimension one, the manifold is clearly flat; in dimension two, since $\Ric = \frac\Scal2\Id$ and $\Scal$ is constant, the manifold is Einstein and the spinor is necessarily a KS.
\end{remark}

We further have a technical lemma that will be needed for the proof of Theorem~\ref{thm:main}:  

\begin{lemma}\label{lem:prod}
    Let $(M,g)$ be a Riemannian spin manifold that is locally isometric to a Riemannian product of two non-Ricci-flat Einstein manifolds. Then, $(M,g)$ does not admit AKS.
\begin{proof} 
Suppose that there exists an open subset of $M$ which is isometric to a Riemannian product $U\times V$ of two non-Ricci-flat Einstein manifolds. We denote the Einstein constants on $U$ and $V$ by $a$ and $b$ respectively. It is clear that, for all $X\in\Gamma(TU)$ and $Y\in \Gamma(TV)$, $\nabla_X Y = 0$ and $\mathrm{R}(X,Y)=0$, and hence $\Ric (X)\in \Gamma(TU)$ and $\Ric (Y) \in \Gamma(TV)$, as well as $\mathcal{R}(X \wedge Y) = 0$.  Arguing by contradiction, assume that $(M,g)$ admits an AKS $\psi$. It follows from equation~\eqref{eq:bas1} that, for all $X\in\Gamma(TU)$ and $Y\in \Gamma(TV)$,
\[
0 = (\lambda^2 ab + \lambda\mu(a+b)+\mu^2)X\wedge Y\cdot \psi = (\lambda a+\mu)(\lambda b +\mu)X\cdot Y\cdot \psi.
\]
If $(\lambda a + \mu)(\lambda b + \mu) \neq 0$, we can conclude that $X \cdot Y \cdot \psi = 0$. Multiplying by $Y \cdot X$ on both sides, we get that $g(X,X) g(Y,Y) \psi = 0$. Since both factors $U$ and $V$ are non-trivial, this implies that $\psi$ has zeros, which is a contradiction. Hence, $(\lambda a + \mu)(\lambda b + \mu) = 0$. Without loss of generality, suppose that $\lambda a + \mu = 0$. Then, for $X \in \Gamma(TU)$, $\nabla_X \psi = (\lambda \Ric (X) + \mu X) \cdot \psi = (\lambda a + \mu) X \cdot \psi = 0$. Hence, by~\eqref{eq:spincurv}, for every $X,Y \in \Gamma(TU)$, $\mathcal{R}(X \wedge Y) \cdot \psi = 0$. Finally, let $(e_1 , \dots , e_m)$ and $(e_{m+1},\dots,e_{n})$ be local orthonormal frames of $TU$ and $TV$ respectively, and let $X \in \Gamma(TU)$. There is a classical identity~\cite[eq. 1.13]{BFGK} 
\[
-\Ric(X) \cdot \psi = \sum_{i=1}^{n} e_i \cdot \mathcal{R}(X \wedge e_i) \cdot \psi. 
\]
However, for all $1 \le i \le m$, $\mathcal{R}(X \wedge e_i) \cdot \psi = 0$, and for all $m+1 \le i \le n$ the product structure yields $\mathcal{R}(X \wedge e_i) = 0$. Hence, $\Ric(X) = 0$, and so $U$ is Ricci-flat. This yields a contradiction, thereby completing the proof.     
\end{proof}
\end{lemma}

We conclude this section by providing an explicit example of a Riemannian spin manifold admitting an AKS that is not a KS, a WKS or an RKS: 
\begin{example}
    Consider the Lie algebra $\mathfrak{su}(2)$, generated by $e_1,e_2,e_3$ with Lie bracket 
    \[
    [e_1,e_2] = \frac{1}{2} e_3 , \quad [e_2,e_3] = 2 e_1 , \quad [e_3 , e_1] = 2 e_2 , 
    \]
    and take the inner product $B$ and orientation determined by declaring $(e_1 , e_2 , e_3)$ to be orthonormal and positively oriented. Note that this metric is not Einstein, and in particular it does not correspond to the round metric on $\mathrm{SU}(2) \cong \mathbb{S}^3$. 
    Then, every invariant spinor on a connected Lie group with Lie algebra $\mathfrak{su}(2)$ equipped with the left-invariant metric induced by $B$ is an AKS with endomorphism $A = - \Ric + \Id$ (see the proof of Theorem~\ref{thm:mainlie}). 
\end{example}

\begin{remark}
    Kim and Friedrich~\cite[Ex. 8.4]{KF} showed the existence of non-trivial WKS on certain constant-scalar-curvature Sasakian deformations of the round sphere $\mathbb{S}^3$. These spinors are, in particular, AKS. Moreover, these deformations are left-invariant, and so these examples of WKS can be seen as special cases of the construction given in the proof of Theorem~\ref{thm:mainlie}. 
\end{remark}

\section{Proofs of the main results}
In this section, we prove the theorems stated in Section~\ref{s:int}.

\subsection{Proof of Theorem~\ref{thm:main}}
For assertion~\ref{it:har}, since $\Ric$ is Codazzi, the left-hand side of equation~\eqref{eq:bas2} vanishes: for every vector field $X$,  
\[
\sum_{i=1}^{n} e_i \wedge \left( \nabla_{e_i} \Ric \right) (X) = \sum_{i=1}^{n} e_i \wedge \left( \nabla_{X} \Ric \right) (e_i) = 0 , 
\]
where the last equality follows from the symmetry of $\nabla_X \Ric$, which follows from the symmetry of $\Ric$. As $\psi$ is nowhere zero, the injectivity of Clifford multiplication implies that $\Ric$ is annihilated by a quadratic polynomial with constant coefficients. Thus, the eigenvalues of $\Ric$ are constant, and $\Ric$ has at most two distinct eigenvalues. 

If $\Ric$ has one eigenvalue, then $M$ is Einstein. Thus, let us assume that $\Ric$ has exactly two distinct constant eigenvalues, denoted by $\alpha$ and $\beta$, with corresponding eigendistributions $E^\alpha$ and $E^\beta$, respectively. We now show that these eigendistributions are parallel. This is a standard argument -- see e.g.~\cite[16.12]{Besse} -- which we include here for completeness. It suffices to prove that $E^\alpha$ is parallel; the same argument applies to $E^\beta$.
For any $X,Y\in E^\alpha$ and $Z\in E^\beta$, using the fact that $\Ric$ is Codazzi we obtain
\[
    0 = g((\nabla_Z \Ric)Y,X) = g((\nabla_Y \Ric)Z,X)=(\alpha-\beta)g(\nabla_Y X,Z),
\]
which implies that $\nabla_Y X \in E^\alpha$. Similarly, for any $X\in E^\alpha$ and $Y,Z\in E^\beta$, we get
\[
    0=g((\nabla_X \Ric)Y,Z) = g((\nabla_Y \Ric)X,Z) = (\alpha-\beta)g(\nabla_Y X,Z).
\]
Hence, $\nabla E^\alpha \subset E^\alpha$, that is, $E^\alpha$ is parallel. 
Consequently, $\Ric$ is parallel. Finally, by the local de Rham decomposition theorem, $M$ locally splits as a Riemannian product of two Einstein manifolds with distinct Einstein constants. Moreover, Lemma~\ref{lem:prod} forces one of the factors to be Ricci-flat. 

If one of the multiplicities is one, we may assume without loss of generality that $\dim E^\alpha = 1$, so that clearly $\alpha = 0$. From~\eqref{eq:bas2} applied to an $\alpha$-eigenvector, it follows that either $\mu = 0$ or $\Scal = \frac{\mu}{\lambda}(1-n)$. In the first case, $\psi$ is an RKS. Assume therefore that $\Scal = \frac{\mu}{\lambda}(1-n)$. Now, from~\eqref{eq:bas2} applied to a $\beta$-eigenvector, we obtain that $\beta = -\frac{1}{\lambda^2}(\tfrac{1}{4}+\lambda\mu)$. But then, as $\alpha$ has multiplicity one, we also have that $\beta = \tfrac{\Scal}{n-1} = -\tfrac{\mu}{\lambda}$. Comparing these two expressions for $\beta$ yields a contradiction.

We now prove assertion~\ref{it:lcf}. By Lemma~\ref{lem:AKSbasic}, the existence of an AKS forces the scalar curvature to be constant. In dimension $2$, this finishes the proof. In dimension $\ge 3$, local conformal flatness together with constancy of the scalar curvature implies that $M$ has harmonic curvature. Now, assertion~\ref{it:har} and~\cite[Prop.~3.6]{HL18} conclude the proof. However, we write it here for completeness. 

Using assertion~\ref{it:har}, the claim follows immediately when $\dim M=3$, as well as when $\dim M\ge 4$ and $M$ is not (locally) a product, since every locally conformally flat Einstein manifold has constant sectional curvature. It remains to consider the case where $\dim M =n\ge 4$ and $M$ is locally isometric to a Riemannian product $U\times V$, where $U$ is an Einstein manifold of dimension $p$ with Einstein constant $u$, and $V$ is a Ricci-flat manifold of dimension $q$. Since $M$ is locally conformally flat, for unit vector fields $X\in\Gamma(TU)$ and $Y\in\Gamma(TV)$ we have 
\[
0=g(R(X,Y)Y,X)=\frac{1}{n-2}\left(\mathrm{ric}(X,X)+\mathrm{ric}(Y,Y)-\frac{\Scal}{n-1}\right)
\]
and hence $u = \frac{\Scal}{n-1}$. As $\Scal = pu$ and $n=p+q$, it follows that $(p+q-1)u = pu$. If $q\ge 2$, this implies $u=0$, and so $M$ is flat. 
If $q=1$, then by restricting the Weyl curvature of $M$ to $U$, we conclude that $U$ has constant sectional curvature $\frac{\Scal}{(n-1)(n-2)}$, while $V$ is flat. 
\qed

\subsection{Proof of Theorem~\ref{thm:p1form}}
    Assume that $M$ admits an AKS $\psi$ satisfying~\eqref{eq:ars} and note that the scalar curvature $\Scal$ is constant by Lemma~\ref{lem:AKSbasic}. Then, the Dirac operator $D$ satisfies $D\psi = -(\lambda\Scal + n\mu)\psi$. Let $\xi$ denote the unit vector field dual to a non-zero parallel one-form. As $\Ric(\xi)=0$, $\nabla_\xi \psi = \mu \,\xi\cdot \psi$. Applying~\cite[Lem.~4.5]{KF} we obtain $D(\nabla_\xi \psi) = \nabla_\xi (D\psi) = -\mu(\lambda\Scal + n\mu)\xi\cdot \psi$. But then, we also obtain $D(\nabla_\xi \psi ) = D(\mu \,\xi\cdot\psi) = \mu D(\xi\cdot \psi) = \mu(\lambda\Scal+(n-2)\mu) \xi \cdot \psi$. Comparing these two yields either $\mu=0$ or $\lambda\Scal + (n-1)\mu = 0$. If $\Scal=0$, then $\mu=0$, and hence $M$ is Ricci-flat by~\eqref{eq:bas3}, so that the spinor is parallel.
    We now assume that $\Scal \ne 0$. It follows that either $\mu=0$, in which case the spinor is an RKS, or $\Scal = \frac{\mu}{\lambda}(1-n)$.
    In the former case, from~\cite[Thm.~A]{I26} we obtain $\lambda^2\ge \frac{n}{4(n-1)\Scal}$. Moreover, equality holds if and only if equality is attained in the Cauchy-Schwarz inequality, which is equivalent to $M$ being Einstein with $\Ric= \frac{\Scal}{n}\Id$, and thus the spinor is a KS with Killing constant $\frac{\lambda\,\Scal}{n}$; that is, $\nabla_X \psi = \frac{\lambda\,\Scal}{n} X \cdot \psi$ for any $X\in \Gamma(TM)$. Now recall the integrability condition for a KS~\cite{BFGK}: $\Ric X = 4(\frac{\lambda\,\Scal}{n})^2(n-1)X$ for all $X\in \Gamma(TM)$. Since $\Ric(\xi)= 0$, we must have $\lambda = 0$, which is a contradiction. Therefore, the equality case cannot occur.
    In the latter case, equation~\eqref{eq:bas3} yields $\lambda^2\|\Ric\|^2 = \mu^2 (n-1) -\frac{1}{4}\Scal$. Applying the Cauchy-Schwartz inequality, $\Scal^2 \le n\|\Ric\|^2$, we obtain $\mu^2 + \frac{n\mu}{4\lambda} \ge  0$. Using $\Scal = \frac{\mu}{\lambda}(1-n)$, this is equivalent to $\mu^2 \ge \frac{n}{4(n-1)}\Scal$. As in the former case, the equality case can be ruled out by a similar argument. Furthermore, if $n=3$, we see that the spinor is a WKS, and by~\cite[Thm.~4.5]{KF} such a spinor cannot exist. This completes the proof.
    \qed

\subsection{Proof of Theorem~\ref{thm:locallyhom}}

Since in dimension $3$ the Riemann curvature tensor is fully determined by $\Ric$, curvature homogeneity is equivalent to the eigenvalues of $\Ric$ being constant. Let $\lambda_1 , \lambda_2 , \lambda_3$ be the three eigenvalues of $\Ric$, and let $(e_1 , e_2 , e_3)$ be an oriented local orthonormal frame such that $\Ric (e_i) = \lambda_i e_i$. Define $h_{ijk} := g\left( \left( \nabla_{e_i} \Ric \right) (e_j) , e_k \right) =(\nabla_{e_i}\mathrm{ric})(e_j,e_k)$ and observe that the symmetry of $\Ric$ implies that for all $i,j,k$
\begin{equation}\label{eq:hsym}
h_{ijk} = h_{ikj} . 
\end{equation}

The proof is divided into three cases:
\begin{itemize}
    \item If $\lambda_1 = \lambda_2 = \lambda_3$, then $(M,g)$ has constant sectional curvature, and hence it is locally homogeneous. 
    \item If $\lambda_1 \neq \lambda_2 = \lambda_3$, local homogeneity is equivalent to the tensor $\nabla \mathrm{ric}$ having constant norm~\cite[Thm.~B and Prop.~2.1]{Yam}. 
    \item If $\lambda_1,\lambda_2,\lambda_3$ are all distinct, to prove local homogeneity it is enough to show that the one-form $P$ defined by 
\[
P_i := \mathrm{ric}^{jk} \left( \nabla_{k} \mathrm{ric} \right)_{ij}
\]
is identically zero~\cite[Lem.~3.1 and Lem.~2.5]{Yam}.
\end{itemize}

We will now show that the existence of an AKS imposes additional symmetry conditions on $h$, which will force $\nabla \mathrm{ric}$ to have constant norm if $\lambda_1 \neq \lambda_2 = \lambda_3$ and $P$ to vanish if $\lambda_1,\lambda_2,\lambda_3$ are all distinct. 

Let $\psi$ be an AKS satisfying~\eqref{eq:ars}. 
Now, let $(i,j,k)$ be an even permutation of $(1,2,3)$ and apply equation~\eqref{eq:bas2}
to $X = e_i$. As we are in dimension $3$, $e_i \cdot e_j \cdot \psi = e_k \cdot \psi$, and hence the left-hand side can be expressed as $V \cdot \psi$, for some vector field $V$. Moreover, the right-hand side is also of the form $W \cdot \psi$ for some vector field $W$. As $\psi$ is nowhere-vanishing, this implies that $V = W$. Equating the $j$-th components of $V$ and $W$ and using that $\lambda \neq 0$, we get that $h_{kii} - h_{iik} = 0$. 
By cyclically permuting the indices, we obtain that for all $i,j$: 
\begin{equation}\label{eq:hAKS1}
h_{jji}=h_{ijj} . 
\end{equation}
Similarly, equating the $i$-th components of $V$ and $W$, we get that
\begin{equation}\label{eq:hAKS2}
h_{jik} - h_{kij} = C_{ijk} , 
\end{equation}
for some constant $C_{ijk}$.  
Note that  
\[
\| \nabla \mathrm{ric} \|^2 = \sum_{i,j,k} h_{ijk}^2 .
\]
By the symmetries~\eqref{eq:hsym},~\eqref{eq:hAKS1} and~\eqref{eq:hAKS2} of $h$, to show that $\nabla \mathrm{ric}$ has constant norm it is enough to show that $h_{123}$ and $h_{ijj}$ are constant for all $i,j$. Indeed, 
\begin{equation}\label{eq:hijj0}
\begin{aligned}
h_{ijj} &= g \left( \left( \nabla_{e_i} \Ric \right) (e_j) , e_j \right) = g \left( \nabla_{e_i} \left(\Ric (e_j) \right) - \Ric \left( \nabla_{e_i} e_j \right) , e_j \right) \\
&= g \left( \left( \lambda_j \Id - \Ric \right) \nabla_{e_i} e_j , e_j \right) = g \left( \nabla_{e_i} e_j , \left( \lambda_j \Id - \Ric \right) e_j \right) = 0. 
\end{aligned}
\end{equation}
Moreover, if $\lambda_1 \neq \lambda_2 = \lambda_3$,  
\begin{equation*}
    \begin{aligned}
        h_{123} &= g \left( \left( \nabla_{e_1} \Ric \right) (e_2) , e_3 \right) = g \left( \left( \nabla_{e_1} \Ric (e_2) \right) - \Ric \left( \nabla_{e_1} e_2 \right) , e_3 \right) \\
        &= g \left( \left( \lambda_2 \Id - \Ric \right) \left( \nabla_{e_1} e_2 \right) , e_3 \right) = (\lambda_2 - \lambda_3) g \left( \nabla_{e_1} e_2 , e_3 \right) = 0 , 
    \end{aligned}
\end{equation*}
and hence $\nabla \mathrm{ric}$ has constant norm. 

We now suppose that $\lambda_1,\lambda_2,\lambda_3$ are all distinct. Observe that, by~\eqref{eq:hsym},~\eqref{eq:hAKS1} and~\eqref{eq:hijj0}, 
\[
P_i = \sum_{j=1}^{3} \lambda_j h_{jij} = \sum_{j=1}^{3} \lambda_j h_{jji} = \sum_{j=1}^{3} \lambda_j h_{ijj} = 0,
\] 
for all $i=1,2,3$, which completes the proof.\qed

\subsection{Proof of Theorem~\ref{thm:mainlie}}

In~\cite{Art}, the first author classified all invariant GKS on $3$-dimensional Lie groups equipped with a left-invariant metric. In particular, it was shown that, given a connected and oriented $3$-dimensional Lie group $G$ with Lie algebra $\mathfrak{g}$ equipped with a left-invariant Riemannian metric $g$ induced from an inner product $B$ on $\mathfrak{g}$, there exists a left-invariant endomorphism $A$ of $TG$ such that every left-invariant spinor $\psi$ satisfies 
\(
\nabla_X \psi = A(X) \cdot \psi, 
\)
for all vector fields $X$. It was shown that, moreover, the $g$-symmetry of $A$ depends only on $\mathfrak{g}$, and not on $B$. Specifically, $A$ is symmetric if and only if $\mathfrak{g}$ is unimodular. Hence, unimodular $3$-dimensional Lie groups are the only ones that admit non-trivial invariant GKS. 

Following~\cite[Sec. 4]{M76}, suppose that $\mathfrak{g}$ is unimodular, and let $(e_1,e_2,e_3)$ be an oriented $B$-orthonormal frame satisfying 
\[
[e_1 , e_2] = c_3 e_3 , \quad [e_2 , e_3] = c_1 e_1 , \quad [e_3 , e_1] = c_2 e_2 , 
\]
for some $c_1 , c_2 ,c_3 \in \R$. Letting  
\[
d_i := \frac{1}{2} ( c_1 + c_2 + c_3 ) - c_i , 
\]
one can see that the Ricci endomorphism satisfies 
\[
\Ric(e_1) = 2 d_2 d_3 e_1 , \quad \Ric(e_2) = 2 d_3 d_1 e_2 , \quad \Ric(e_3) = 2 d_1 d_2 e_3 . 
\]
Finally, computing as in~\cite{Art}, the endomorphism $A$ is given by $A(e_i) = (d_i / 2) \, e_i$, for all $i = 1 , 2 , 3$. Therefore, $A$ is of the form $\lambda \Ric + \mu \Id$ and not of the form $\lambda \Ric$ or $\mu \Id$ if and only if 
\begin{equation}\label{eq:conditions}
\mathrm{rank} \begin{pmatrix}
d_2 d_3 & 1\\
d_3 d_1 & 1 \\
d_1 d_2 & 1 
\end{pmatrix} = \mathrm{rank} \begin{pmatrix}
d_2 d_3 & d_1\\
d_3 d_1 & d_2 \\
d_1 d_2 & d_3 
\end{pmatrix} = 2 \qquad \text{and} \qquad \det \begin{pmatrix}
d_2 d_3 & 1 & d_1 \\
d_3 d_1 & 1 & d_2 \\
d_1 d_2 & 1 & d_3
\end{pmatrix} = 0 .  
\end{equation}

Up to reordering, the conditions in~\eqref{eq:conditions} hold if and only if 
\[
d_3 d_1 \neq d_1 d_2 \, , \quad d_1 = d_2 \, , \quad \text{and} \quad  d_1^2 \neq d_3^2 \, , 
\]
because the determinant is of Vandermonde type. One easily sees that these conditions are equivalent to 
\[
0 \neq d_1 = d_2 \neq d_3 \quad  \text{and} \quad d_1 \neq -d_3 \, .
\]
In terms of the Milnor constants $c_i$, this is equivalent to 
\[
c_1 = c_2 \neq c_3 \quad \text{and} \quad c_1 , c_2 , c_3 \neq 0 \, . 
\]
By~\cite[p.~307]{M76}, this happens precisely on the Lie algebras $\mathfrak{su}(2)$ and $\mathfrak{sl}(2,\mathbb{R})$. \qed

\bibliographystyle{alphaurl}
\bibliography{references.bib}

\end{document}